\documentclass[reqno,11pt]{amsart}

\usepackage{amssymb,amsthm}
\usepackage{amsmath}
\usepackage{amsfonts}
\usepackage{dsfont}

\usepackage[a4paper,  margin=3cm]{geometry}

\usepackage[active]{srcltx}
\makeatletter\@addtoreset{equation}{section}\makeatother

\newtheorem{theorem}{Theorem}[section]
\newtheorem{corollary}[theorem]{Corollary}
\newtheorem{lemma}[theorem]{Lemma}

\newtheorem{proposition}[theorem]{Proposition}
\newtheorem{assumption}[theorem]{Assumption}

\newtheorem{remark}[theorem]{Remark}
\numberwithin{equation}{section}

\title[Stochastic branching at the
edge]{Stochastic branching at the edge: Individual-based modeling of
tumor cell proliferation\footnote{Financially supported by National
Science Centre, Poland (NCN), grant 2017/25/B/ST1/00051}}

\author{ Yuri  Kozitsky}

\address{Instytut Matematyki, Uniwersytet Marii Curie-Sk{\l}odowskiej, 20-031 Lublin, Poland}
\email{jkozi@hektor.umcs.lublin.pl}

\begin{document}


\begin{abstract}
An individual-based model of stochastic branching is proposed and
studied, in which point particles drift in
$\bar{\mathds{R}}_{+}:=[0,+\infty)$ towards the origin (edge) with
unit speed, where each of them splits into two particles that
instantly appear in $\bar{\mathds{R}}_{+}$ at random positions.
During their drift the particles are subject to a random
disappearance (death). The model is intended to capture the main
features of the proliferation of tumor cells, in which trait $x\in
\bar{\mathds{R}}_{+}$ of a given cell is time to its division and
the death is caused by therapeutic factors. The main result of the
paper is proving the existence of an honest evolution of this kind
and finding a condition that involves the death rate and cell cycle
distribution parameters, under which the mean size of the population
remains bounded in time.

\end{abstract}

\maketitle

\section{Introduction}
One of the most natural applications of semigroups of bounded
positive operators in $L^1$-like spaces \cite{hon,Banasiak,TV} is
the description of the evolution of probability densities, widely
used in population biology, genetics, medical sciences, etc. Among
the processes  which have multiple applications one might
distinguish branching \cite{KA}. In this paper, we propose and study
an individual-based model that can describe the proliferation of
tumor cells, cf. \cite[Sect. 2.2]{Kimm}. Here `individual-based'
means that each single member of the population is taken into
account in an explicit way. This is in contrast to macroscopic
models where populations are described in terms of aggregate
parameters such as density, cf. e.g., \cite{Mee,Rot}, which might be
considered as an advantage of the theory.  In the proposed model,
each member of a finite population of particles (assuming cells) is
characterized by a random trait $x\in [0,+\infty)$ -- time to its
branching (fission or division). Then one of the basic acts of the
dynamics is \emph{aging} - diminishing of the traits with unit
speed. At point $x=0$, the particle divides into two progenies with
randomly distributed traits -- the second basic act of its dynamics.
Finally, during the whole lifetime each particle is subject to a
random death, which in the case of tumor cells can be caused by
therapeutic factors. The main questions concerning this model which
we address here are: (a) can one expect (and under which conditions)
that the population dynamics is \emph{honest}; (b) what might be a
condition for the boundedness in time of the population mean size.
The mentioned honesty of the dynamics means that the population
remains in time almost surely finite (no explosion). In the language
of semigroups of positive operators, cf. \cite{hon,Banasiak}, this
means that the semigroup of operators mapping the initial
probability distribution of the population traits on those
corresponding to $t>0$ is \emph{stochastic}. The mentioned
boundedness condition ought to involve the death rate and cell cycle
distribution parameters. Its practical meaning might be estimating
at which level of the therapeutic pressure the tumor cell population
stops growing ad infinitum. This aspect of the theory is indeed
practical since, for various kinds of tumors, the cell cycle
distributions and their parameters are known, see \cite{Marz,Kimm}
and also \cite{Gabriel,Tyson,Ye}.

The rest of the paper is organized as follows. In Section 2, we
introduce the necessary mathematical framework and then two models,
of which the second one is the principal model mentioned above. The
introduction of this model is preceded by a careful investigation of
its `mild' version, in which the particles instead of fission just
disappear at the edge. This turns useful in the subsequent study of
the principal model. In Section 2, we also formulate the main result
-- Theorem \ref{1tm}. Its proof is performed in Section 3, whereas
concluding remarks are placed in Section 4.

\section{The Model and the Result}

We begin by providing necessary notions and facts. Then we introduce
an auxiliary model, the advantage of which is that it is soluble.
This allows us to clarify a number of properties of the principal
model introduced afterwards. Next, we formulate the result as
Theorem \ref{1tm}.

\subsection{Preliminaries}
In this work, we use the following standard notations:
$\mathds{R}_{+} = (0,+\infty)$, $\bar{\mathds{R}}_{+} =[0,+\infty)$,
$\mathds{N}_0=\mathds{N}\cup \{0\}$, $\mathds{N}$ stands for the set
of positive integers. For a Banach space, $(\mathcal{E},
\|\cdot\|_{\mathcal{E}})$, with a cone of positive elements,
$\mathcal{E}^{+}$, a $C_0$-semigroup $S=\{S(t)\}_{t\geq 0}$ of
bounded linear operators $S(t):\mathcal{E}\to \mathcal{E}$ is called
sub-stochastic (resp. stochastic) if, for each $t\geq 0$, $S(t) :
\mathcal{E}^{+} \to \mathcal{E}^{+}$ and $\|S(t)
u\|_{\mathcal{E}}\leq 1$ (resp. $\|S(t) u\|_{\mathcal{E}}= 1$)
holding for all $u\in \{ u \in \mathcal{E}^{+}: \|u\|_{\mathcal{E}}
=1\}$.

By $\Gamma$ we denote the set of all finite subset of
$\bar{\mathds{R}}_{+}$. Its elements are finite {\it
configurations}.  This set is equipped with the weak topology, see
\cite{DV1}, which is metrizable in such a way that the corresponding
metric space is separable and complete. Namely, a sequence,
$\{\gamma_n\}_{n\in \mathds{N}}\subset \Gamma$, is convergent in
this topology to some $\gamma\in \Gamma$ if
\[
 \sum_{x\in \gamma_n} g(x) \to \sum_{x\in \gamma} g(x),
\]
that holds for all bounded continuous functions
$g:\bar{\mathds{R}}_{+} \to \mathds{R}$. Let $\mathcal{B}(\Gamma)$
be the corresponding Borel $\sigma$-field. Then $(\Gamma,
\mathcal{B}(\Gamma))$ is a standard Borel space. A function,
$f:\Gamma \to \mathds{R}$, is measurable if and only if there exists
a family of symmetric Borel functions $f^{(n)}:
\bar{\mathds{R}}_{+}^n \to \mathds{R}$, $n\in \mathds{N}_0$ such
that $f(\varnothing ) = f^{(0)}\in \mathds{R}$ and
\begin{equation}
 \label{0}
f(\{x_1, \dots , x_n\}) = f^{(n)} (x_1, \dots , x_n), \qquad n\in
\mathds{N}.
 \end{equation}
Note that $f^{(n)}$ are defined up to their values at points of
coincidence $x_i=x_j$. However, this makes no problem as we are
going to deal with $L^1$-like spaces, elements of which are defined
up to sets of (Lebesgue)-measure zero.

To simplify our notations, in expressions like $\gamma \cup x$,
$x\in \bar{\mathds{R}}_{+}$ we consider $x$ as a single-element
configuration $\{x\}$. The Lebesgue-Poisson measure $\lambda$ on
$(\Gamma, \mathcal{B}(\Gamma))$ is defined by the integrals
\begin{eqnarray}
  \label{L1}
&  &  \int_{\Gamma} f(\gamma) \lambda (d \gamma) = f^{(0)} +
  \sum_{n=1}^\infty \frac{1}{n!} \int_{\bar{\mathds{R}}_{+}^n}
  f^{(n)} (x_1 , \dots , x_n) d x_1 \cdots x_n ,
\end{eqnarray}
holding for all bounded measurable $f:\Gamma\to \mathds{R}$. Such
integrals have the following evident property which we will use
throughout the whole paper
\begin{equation}
  \label{La}
  \int_{\Gamma} \left(\sum_{\xi \subset \gamma} f(\gamma, \xi)
  \right) \lambda ( d \gamma) = \int_{\Gamma}\int_{\Gamma} f
  (\gamma\cup \xi, \xi) \lambda ( d \gamma) \lambda ( d \xi).
\end{equation}
Let $h:\Gamma\to \mathds{R}_{+}$ be separated away from zero and
such that
\begin{equation}
  \label{10}
  \int_{\Gamma} |f(\gamma)| \lambda (d \gamma) \leq \int_{\Gamma}h(\gamma) |f(\gamma)| \lambda (d
  \gamma)=: \|f\|_h,
\end{equation}
holding for all $f\in\mathcal{X}:= L^1 (\Gamma, d\lambda)$. Then we
set $\mathcal{X}_h= L^1 (\Gamma,h d\lambda)$ and equip it with the
norm defined in (\ref{10}). By $\|\cdot \|$ we will denote the norm
of $\mathcal{X}$, and $\mathcal{X}^{+}$, $\mathcal{X}^{+}_h$ will
stand for the cones of positive elements of $\mathcal{X}$ and
$\mathcal{X}_h$, respectively. Note that
\begin{equation}
 \label{La1}
 \mathcal{X}_h \hookrightarrow \mathcal{X}, \qquad \mathcal{X}^{+}_h \hookrightarrow \mathcal{X}^{+},
\end{equation}
where $\hookrightarrow$ denotes continuous embedding. Clearly,
$\mathcal{X}_h$ and $\mathcal{X}^{+}_h$ are dense in $\mathcal{X}$
and $\mathcal{X}^{+}$, respectively.

For a given $n\in \mathds{N}$, by $W^{1,1}(\mathds{R}_{+}^n)$ we
denote the standard Sobolev space \cite{Maz}, whereas $W^{1,1}_{\rm
s}(\mathds{R}_{+}^n)$ will stand for its subset consisting of all
symmetric $u$, i.e., such that $u(x_1 , \dots , x_n) =
u(x_{\sigma(1)}, \dots x_{\sigma(n)})$ holding for all permutations
$\sigma \in \varSigma_n$.

\begin{remark}
 \label{1rk}
It is known, cf. \cite[Theorem 1, page 4]{Maz}, that each element of
$W^{1,1}_{\rm s}(\mathds{R}_{+}^n)$ -- as an equivalence class --
contains a unique (symmetric) $u:\bar{\mathds{R}}^n_{+} \to
\mathds{R}$ such that
\begin{itemize}
\item[(a)] for Lebesgue-almost all $(x_1, \dots, x_{n-1})$, the map $\bar{\mathds{R}}_{+}\ni y \mapsto u(y, x_1 \dots , x_{n-1})$ is continuous and its restriction to
$\mathds{R}_{+}$ is absolutely continuous;
\item[(b)] the following holds
\begin{equation}
  \label{8}
  \int_{\mathds{R}_{+}^n}\left| \frac{\partial}{\partial x_1} u(x_1 , \dots , x_n)\right| dx_1 \cdots d
  x_n < \infty.
\end{equation}
\end{itemize}
In the sequel, we will mean this function $u$ when speaking of a
given element of $W^{1,1}_{\rm s}(\mathds{R}_{+}^n)$.
\end{remark}
For such $u$, set
\begin{equation}
  \label{9}
  k_u (x) = \int_{\mathds{R}_{+}^{n-1}} u(x, x_1, \dots , x_{n-1})
  dx_1 \cdots d x_{n-1}, \qquad x \in \bar{\mathds{R}}_{+}.
\end{equation}
Then $k_u \in W^{1,1}(\mathds{R}_{+})$.

As mentioned above, see (\ref{0}), each measurable $f:\Gamma \to
\mathds{R}$ defines symmetric Borel functions $f^{(n)}:
\bar{\mathds{R}}^n_{+}\to \mathds{R}$, $n\in \mathds{N}_0$. Let
$f\in \mathcal{X}$ be such that each $f^{(n)}$ belongs to the
corresponding $W^{1,1}_{\rm  s}(\mathds{R}^n_{+})$. Set
\begin{eqnarray}
  \label{11}
  (D f)^{(n)}(x_1, \dots , x_n) & = & \sum_{j=1}^n
  \frac{\partial}{\partial x_j} f^{(n)} (x_1, \dots , x_n) \\[.2cm]
  & = & \frac{d}{dt} f^{(n)} (x_1+t , \dots , x_n+t )|_{t=0}.
  \nonumber
\end{eqnarray}
Then by $\mathcal{W}$ we denote the subset of $\mathcal{X}$
consisting of all those $f$ for which $f^{(n)}\in W^{1,1}_{\rm s}
(\mathds{R}^n)$ and the following holds
\begin{eqnarray}
  \label{12}
  \| D f\|&:=& \sum_{n=1}^\infty \frac{1}{n!} \int_{\mathds{R}_{+}^n} \sum_{j=1}^n \left| \frac{\partial}{\partial x_j}
  f^{(n)}
  (x_1 ,\dots , x_n) \right| d x_1 \cdots dx_n \\[.2cm] \nonumber & = &  \sum_{n=0}^\infty \frac{1}{n!} \int_{\mathds{R}_{+}^{n+1}} \left| \frac{\partial}{\partial x}
  f^{(n+1)}
  (x, x_1 ,\dots , x_n) \right| d x d x_1 \cdots dx_n <\infty.
\end{eqnarray}
Note that the key point here is the convergence of the series.
 For $\gamma\in \Gamma$, set
$\gamma_t= \{x+t: x\in\gamma\}$, i.e., $\gamma_t$ is a shift of
$\gamma$. Obviously, $\gamma_t\in \Gamma$ for $t>0$. In the sequel,
we will use such shifts also with negative $t$ in the situations
where all $x+t \geq 0$. By (\ref{11}) and (\ref{12}), for $f\in
\mathcal{W}$, we have
\begin{gather}
  \label{12a}
  (Df)(\gamma) =  \frac{d}{dt} f(\gamma_t)|_{t=0}, \\[.2cm] \nonumber
  f(\gamma_t)  =  f(\gamma) + \int_0^t (D f) (\gamma_\tau)  d\tau.
\end{gather}
For $f\in \mathcal{W}$, we then set
\begin{equation}
 \label{12A}
\|f\|_{\mathcal{W}} = \|f \| + \|Df\|.
\end{equation}
\begin{proposition}
  \label{1pn}
The set $\mathcal{W}$ equipped with the norm defined in (\ref{12A})
is a Banach space. Thus, the linear operator $(D, \mathcal{W})$
defined on $\mathcal{X}$ in
  (\ref{11}) and (\ref{12}) is closed.
\end{proposition}
\begin{proof}
Let $\{f_m\}_{m\in \mathds{N}}\subset \mathcal{W}$ be a Cauchy
sequence in $\|\cdot \|_{\mathcal{W}}$. For each $f_m$, let
$f^{(n)}_m$, $n\in \mathds{N}_0$ be defined as in (\ref{0}). Then,
for each $n\in \mathds{N}_0$, $\{f_m^{(n)}\}_{m\in
\mathds{N}}\subset W^{1,1}_{\rm s}(\mathds{R}^n)$ is a Cauchy
sequence in the Sobolev space $W^{1,1}(\mathds{R}^n)$. Let $f^{(n)}$
be its limit, which exists as $W^{1,1}(\mathds{R}^n)$ is complete.
Clearly, $f^{(n)}$ is symmetric, i.e., $f^{(n)}\in W^{1,1}_{\rm
s}(\mathds{R}^n)$. Since $\{f_m\}_{m\in \mathds{N}}$ is a Cauchy
sequence in $\mathcal{X}$, it converges there to some $f\in
\mathcal{X}$ such that its $f^{(n)}$ are the limits of the sequences
$\{f_m^{(n)}\}_{m\in \mathds{N}}$ as just discussed. By Remark
\ref{1rk} these $f^{(n)}$ satisfy (\ref{8}); hence, for all $N\in
\mathds{N}$, the following holds
\begin{equation*}
\|f\|_{N}:= \sum_{n=1}^N \frac{1}{n!} \sum_{j=1}^n \left|
\frac{\partial}{\partial x_j} f^{(n)} (x_1 , \dots , x_n)\right| d
x_1 \cdots d x_n < \infty.
\end{equation*}
Let us then show that the sequence $\{\|f\|_{N}\}_{N\in \mathds{N}}$
is bounded, and thus $f$ lies in $\mathcal{W}$. Set $C =
\sup_{m}\|Df_m\|$. Then, for each $m\in \mathds{N}$, by the triangle
inequality we have that
\[
\|f\|_{N} \leq C + \|f-f_m\|_{N} \leq 2 C.
\]
The second inequality holds for a fixed $N$ and $m>m_N$ for an
appropriate $m_N$. This yields the proof of the first part of the
statement. Then the closedness follows by the fact that (\ref{12A})
is exactly the graph norm of $(D,\mathcal{W})$.
\end{proof}
\begin{corollary}
  \label{1pnco}
The operator $(D,\mathcal{W})$ is the generator of a sub-stochastic
semigroup, $S_0 = \{S_0(t)\}_{t\geq 0}$, on $\mathcal{X}$ such that
$(S_0(t) f)(\gamma) = f(\gamma_t)$. Hence, $\|S_{0}(t) f - f\| \to
0$ as $t\to 0^{+}$ for each $f\in\mathcal{X}$.
\end{corollary}
\begin{proof}
Clearly, $(S_0(t)f)(\gamma_s)= f(\gamma_{t+s})$ and the map
$f\mapsto S_0(t)f$ is positivity-preserving and such that
$\|S_0(t)f\| \leq \|f\|$. Then $S_0$ is a positive semigroup
generated by $(D,\mathcal{W})$, see (\ref{12a}). Set
 \begin{equation*}
  (R_\varkappa (D)f)(\gamma) = \int_{0}^{+\infty} e^{-\varkappa t} f(\gamma_t) d t, \qquad \varkappa>0.
 \end{equation*}
Then $\| R_\varkappa(D) f\| \leq {1}/{\varkappa}$, by which and
Proposition \ref{1pn}  $S_0$ is a $C_0$-semigroup, see, e.g.,
\cite[Theorem 3.1, page 8]{Pazy}. This completes the proof.
\end{proof}

\subsection{A soluble model}

As mentioned above, our principal model is a modification of another
model, the main advantage of which is that it is in a sense soluble.
We will use this fact in studying the principal model below. This
soluble model describes the following process. A finite cloud of
point particles is distributed over $\bar{\mathds{R}}_{+}$. Each
particle in the cloud moves towards the origin with unit speed, and
disappears at $x=0$. The states of the cloud are probability
measures on $\mathcal{X}$, which we assume to be absolutely
continuous with respect to $\lambda$ defined in (\ref{L1}). Set
$\Gamma_t= \{\gamma \in \Gamma: \gamma \subset [0,t)\}$, and
$\Gamma_t^c= \{\gamma \in \Gamma: \gamma \subset [t,+\infty)\}$,
$t\geq 0$. Note that $\Gamma_t \cup \Gamma_t^c \neq \Gamma$. Let
$f_t$ be the density (Radon-Nikodym derivative) of the state at time
$t$, and $f$ be the density of the initial state. As described
above, the cloud undergoes the evolution according to the following
formula
\begin{equation}
  \label{2}
  f_t (\gamma) = \int_{\Gamma_t} f(\gamma_t \cup \xi) \lambda (
  d\xi) = \int_{\Gamma}f(\gamma_t \cup \xi)\mathds{1}_{\Gamma_t}(\xi) \lambda (
  d\xi).
\end{equation}
Here and in the sequel, by $\mathds{1}$ we denote the corresponding
indicator. Let us show that $\{f_t\}_{t\geq 0}$ has the flow
property
\begin{equation}
  \label{3}
  f_{t +s}(\gamma) = \int_{\Gamma_t} f_s(\gamma_t \cup \xi) \lambda (
  d\xi), \qquad s, t >0.
\end{equation}
To this end we write
\[
 \mathds{1}_{\Gamma_t} (\xi) = \prod_{x\in \xi} \mathds{1}_{[0,t)} (x),
\]
express $f_s$ in the right-hand side of (\ref{3}) by (\ref{2}), and
then obtain
\begin{eqnarray*}
{\rm RHS(\ref{3})} & = & \int_{\Gamma}\int_{\Gamma}
\left( \prod_{x\in \eta}\mathds{1}_{[0,s)} (x) \right) \left( \prod_{y\in \xi}\mathds{1}_{[s,s+t)} (y) \right) f(\gamma_{t+s} \cup\xi \cup \eta) \lambda ( d \xi) \lambda ( d \eta)  \qquad \qquad \\[.2cm] & = & \int_{\Gamma} f (\gamma_{t+s} \cup \eta) \left[\sum_{\xi \subset \eta}
\left( \prod_{y\in \xi}\mathds{1}_{[s,s+t)} (y)\right)\left(
\prod_{x\in \eta\setminus \xi}\mathds{1}_{[0,s)} (x)\right)\right]
\lambda ( d \eta) \\[.2cm] & = & \int_{\Gamma} f (\gamma_{t+s} \cup \eta) \left[ \prod_{x\in \eta}
\left( \mathds{1}_{[0,s)} (x) + \mathds{1}_{[s,s+t)} (x)\right)\right] \lambda ( d \eta) \\[.2cm]
& = & \int_{\Gamma_{t+s}} f (\gamma_{t+s} \cup \eta) \lambda ( d
\eta) = {\rm LHS(\ref{3})}.
\end{eqnarray*}
Let $|\gamma|$ denote the number of points in $\gamma\in \Gamma$.
For $f$ as in (\ref{2}), we define
\begin{equation*}
 N_l := \int_\Gamma |\gamma|^l f(\gamma) \lambda ( d \gamma), \qquad l\in \mathds{N}_0.
\end{equation*}
Note that $N_1$ is just the expected number of points in the cloud
as time $t=0$. Note also that $N_0=1$ in view of the normalization
of $f$. Let us prove that
\begin{equation}
 \label{3b}
 N_l (t):= \int_\Gamma |\gamma|^l f_t(\gamma) \lambda ( d \gamma) \leq N_l.
\end{equation}
Indeed, by (\ref{2}) we have
\begin{eqnarray}
\label{3c} N_l (t)& = & \int_\Gamma \int_{\Gamma} |\gamma|^l
f(\gamma_t \cup \xi) \mathds{1}_{\Gamma_t} (\xi) \lambda ( d \xi)
\lambda ( d \gamma)
\\[.2cm]\nonumber & = & \int_{\Gamma_t^c} \int_{\Gamma} \left|\gamma_{-t}\right|^l f(\gamma\cup \xi) \mathds{1}_{\Gamma_t}
(\xi) \lambda ( d \xi) \lambda ( d \gamma) \\[.2cm] \nonumber
& = & \int_{\Gamma} \int_{\Gamma} \left|\gamma_{-t}\right|^l
f(\gamma\cup \xi) \mathds{1}_{\Gamma_t} (\xi)\mathds{1}_{\Gamma_t^c}
(\gamma) \lambda ( d \xi) \lambda ( d \gamma)
\\[.2cm] \nonumber & = & \int_{\Gamma} f (\gamma)
\left( \sum_{\xi \subset \gamma} \left|(\gamma\setminus
\xi)_{-t}\right|^l \mathds{1}_{\Gamma_t} (\xi)
\mathds{1}_{\Gamma_t^c} (\gamma \setminus \xi)\right) \lambda ( d
\gamma).
\end{eqnarray}
For $t>0$,  we write $\gamma\in \Gamma$ in the form $\gamma =
\gamma_1^t \cup \gamma_2^t$ with $\gamma_1^t= \gamma \cap [0,t)$.
Then the sum in the last line of (\ref{3c}) has only one nonzero
term corresponding to $\xi = \gamma_1^t$. That is,
\[
 N_l(t) = \int_{\Gamma} |(\gamma^t_2)_{-t}|^l f(\gamma) \lambda (d\gamma) \leq \int_{\Gamma}|\gamma|^l f(\gamma) \lambda (d\gamma) = N_l,
\]
that yields (\ref{3b}). The latter yields also $N_0(t) = 1$, i.e.,
the map $f\mapsto f_t$ defined in (\ref{2}) preserves the norm.
Notably, $(\gamma_2^t)_{-t}$ is the part of the initial cloud that
remains after time $t$, shifted towards the origin. Its expected
cardinality thus cannot be bigger than that of the initial cloud,
that is reflected in the latter estimate.
\begin{proposition}
 \label{1apn}
 For each $f\in\mathcal{X}$, it follows that $\|f_t - f\|\to 0$ as $t\to 0^{+}$, where $f_t$ and $f$ are related to each other by (\ref{2}).
\end{proposition}
\begin{proof}
Clearly, it is enough to prove the statement for positive $f$ only.
By (\ref{2}) we have
\begin{eqnarray}
\label{3d}
 \|f_t - f\| & = & \int_{\Gamma} \left|f(\gamma) - \int_{\Gamma_t} f(\gamma_t \cup \xi) \lambda ( d \xi) \right| \lambda ( d \gamma)\\[.2cm] \nonumber & \leq & \int_{\Gamma} \left|f(\gamma) - f(\gamma_t) \right| \lambda ( d \gamma) \\[.2cm] \nonumber & + & \int_{\Gamma} \int_{\Gamma} f(\gamma_t \cup \xi) \mathds{1}_{\Gamma_t} (\xi) \chi(\xi) \lambda ( d \xi ) \lambda ( d \gamma) \\[.2cm] \nonumber & =: & I_1(t) + I_2 (t).
\end{eqnarray}
Here $\chi(\xi) = 0$ whenever $\xi = \varnothing$, and $\chi(\xi) =
1$ otherwise. By Corollary \ref{1pnco} we have that $I_1 (t) \to 0$
as $t\to 0^{+}$. To estimate $I_2(t)$ we proceed similarly as in
deriving (\ref{3c}). That is,
\begin{eqnarray}
\label{2e}
I_2(t) & = & \int_{\Gamma_t^c} \left( \int_{\Gamma} f(\gamma\cup\xi) \mathds{1}_{\Gamma_t} (\xi) \chi(\xi) \lambda ( d \xi)  \right)\lambda ( d \gamma)  \\[.2cm] \nonumber & = & \int_{\Gamma} \int_{\Gamma} f(\gamma\cup\xi) \mathds{1}_{\Gamma_t} (\xi) \chi(\xi) \mathds{1}_{\Gamma_t^c} (\gamma)\lambda ( d \xi) \lambda ( d \gamma) \\[.2cm] \nonumber & = & \int_{\Gamma} f(\gamma) \left( \sum_{\xi\subset \gamma} \mathds{1}_{\Gamma_t} (\xi) \chi(\xi) \mathds{1}_{\Gamma_t^c} (\gamma\setminus \xi) \right) \lambda ( d \gamma) \\[.2cm] \nonumber & = &  \int_{\Gamma} f(\gamma) \chi(\gamma^t_1) \lambda ( d \gamma).
\end{eqnarray}
Then $I_2 (t) \to \mu(\Gamma_0)$ where $\mu(d\gamma) = f(\gamma)
\lambda ( d \gamma)$ is the initial state and $\Gamma_0=\{\gamma\in
\Gamma: 0\in \gamma \}$. It is, however, obvious (see also the proof
of Proposition \ref{1bpn} below) that $\mu(\Gamma_0)=0$, that
completes the proof.
\end{proof}
Let $\mathcal{X}^l$, $l\in \mathds{N}$ stand for the Banach space
$\mathcal{X}_h$ with $h(\gamma)= 1+|\gamma|^l$.
\begin{corollary}
  \label{2pnco}
There exists a unique stochastic semigroup, $S^0=\{S^0(t)\}_{t\geq
0}$, on $\mathcal{X}$ such that, for each $f\in \mathcal{X}$,  $f_t=
S^0(t)f$, where $f_t$ and $f$ are the same as in (\ref{2}). The
semigroup $S^0$ leaves invariant each $\mathcal{X}^l$, $l\in
\mathds{N}$.
\end{corollary}
\begin{proof}
The semigroup property of $S^0$ follows by (\ref{3}). Its strong
continuity follows by Proposition \ref{1apn}, whereas the property
$S^0(t) : \mathcal{X}^l\to \mathcal{X}^l$ follows by (\ref{3b}).
\end{proof}
For each $f\in \mathcal{W}$ and $\lambda$-almost all $\gamma\in
\Gamma$, we know that the map $x \mapsto f(\gamma\cup x)$ is
continuous on $\bar{\mathds{R}}_{+}$ and absolutely continuous on
$\mathds{R}_{+}$, see Remark \ref{1rk}. Set
\begin{equation}
  \label{2a}
\mathcal{V}= \left\{ f \in \mathcal{W}: \int_{\Gamma} |f(\gamma\cup
0)|\lambda (d\gamma)<\infty \right\}.
\end{equation}
Note that, for each $x\in \mathds{R}_{+}$ and $f \in \mathcal{V}$,
\begin{equation}
  \label{2A}
\int_{\Gamma} |f(\gamma\cup x)|\lambda (d\gamma)\leq C_f <\infty,
\end{equation}
with an appropriate $C_f>0$, independent of $x$. Indeed, by
(\ref{12}) we have
\begin{gather}
\label{0b} \int_{\Gamma} |f(\gamma\cup x)- f(\gamma\cup 0)|\lambda
(d\gamma) \leq \int_{\Gamma} \left( \int_{\mathds{R}_{+}}
\left|\frac{\partial }{\partial x} f (\gamma\cup x)\right| dx
\right) \lambda ( d \gamma) \leq \| Df\|.
\end{gather}
Then the proof of (\ref{2A}) follows by the definition of
$\mathcal{V}$ and the triangle inequality. For $f\in \mathcal{V}$,
let $f^{(n)}$ be as in (\ref{0}) and $k_{f^{(n)}}$, see Remark
\ref{1rk},  be as in (\ref{9}) for this $f^{(n)}$. In view of
(\ref{2A}), one can define
\begin{equation}
\label{9A}
 k_f (x) = \int_{\Gamma} f(\gamma \cup x) \lambda ( d \gamma).
\end{equation}
Then the map $\bar{\mathds{R}}_{+}\ni x\mapsto k_f\in
\bar{\mathds{R}}_{+}$ is continuous and locally integrable, cf.
\cite[Sect. 1.1.2, pages 2,3]{Maz}. For $0\leq a< b<\infty$,
\[
 \int_a^b k_f(x) dx
\]
is the expected number of points with traits in the interval $[a,b]$
in the corresponding state. A priory $k_f$ need not be integrable on
the whole $\bar{\mathds{R}}_{+}$.

Define
\begin{equation}
  \label{2b}
  (L^0 f) (\gamma) = (D f)(\gamma) + f(\gamma\cup 0), \qquad f \in
  \mathcal{V}.
\end{equation}
Clearly, $L^0:\mathcal{V}\to \mathcal{X}$, in view of which we
introduce the following norm
\begin{equation}
 \label{2B}
 \|f\|_{\mathcal{V}} = \|f\| + \|Df\| + \int_{\Gamma} |f(\gamma\cup 0)|
 \lambda ( d \gamma).
\end{equation}
In the statement below, we will use the set $\mathcal{V}'$
consisting of all those $f\in \mathcal{V}$ which have the following
two properties: (a) for each $x\in \bar{\mathds{R}}_{+}$, the map
$\gamma \mapsto f(\gamma\cup x)$ is in $\mathcal{W}$; (b) for each
$x\in \bar{\mathds{R}}_{+}$,
\[
 \int_{\Gamma} |f(\gamma \cup \{ x, 0\})| \lambda(d \gamma) < \infty.
\]
Let us prove that
\begin{equation}
 \label{2C}
 \mathcal{V}\subset \overline{\mathcal{V}'},
\end{equation}
where the closure is taken in $\|\cdot\|_{\mathcal{V}}$. For $f\in
\mathcal{V}$ and $m\in \mathds{N}$, let $f_m$ be such that
$f_m^{(n)} = f^{(n)}$, $n\leq m$, and $f_m^{(n)} \equiv 0$ for
$n>m$, cf. (\ref{0}). Since each $f^{(n)}$ is in $W^{1,1}_{\rm
s}(\mathds{R}_{+}^n)$, we have that $\{f_m\}_{m}\subset
\mathcal{V}'$, see Remark \ref{1rk}. At the same time
$\|f-f_m\|_{\mathcal{V}} \to 0$ as $m\to +\infty$. Indeed,
\begin{gather*}
 \|f-f_m\|_{\mathcal{V}} = \|f-f_m\| + \|D(f-f_m)\| + \sum_{n=m+1}^\infty \frac{1}{n!} \int_{\bar{\mathds{R}}_{+}^n} \left|f^{(n)} (0, x_1 , \dots , x_n)\right| d x_1 \cdots d x_n.
\end{gather*}
All the three terms of the right-hand side are the remainders of
convergent series, that eventually yields (\ref{2C}).
\begin{proposition}
  \label{1bpn}
It follows that $\mathcal{V}=\mathcal{W}$ and the semigroup $S^0$ as
in Corollary \ref{2pnco} is generated by $(L^0,\mathcal{W})$.
\end{proposition}
\begin{proof}
First we prove that, for all $f\in \mathcal{V}$, it follows that
\begin{equation}
 \label{2D}
\left\|\frac{1}{t} (f_t - f) - L^0f \right\| \to 0, \qquad t\to
0^{+}.
\end{equation}
Clearly, it is enough to show this for $f\in \mathcal{V}^{+}:=
\mathcal{V}\cap \mathcal{X}^{+}$ only. Similarly as in (\ref{3d})
and then (\ref{2e}), for such $f$ we obtain
\begin{gather}
  \label{2c}
  f_t (\gamma) - f(\gamma) = f (\gamma_t) - f(\gamma) +
  \int_{\Gamma} f(\gamma_t \cup \xi) \mathds{1}_{\Gamma_t} (\xi)
  \chi(\xi) \lambda (d\xi) \\[.2cm] \nonumber =: t (D f)
  (\gamma_\tau)+ t F_t(\gamma),
\end{gather}
for some $\tau\in [0,t)$, see (\ref{12a}). Then to prove (\ref{2D})
it suffices to show that
\begin{equation}
  \label{2z}
 I(t):= \int_{\Gamma} | F_t(\gamma) - f(\gamma\cup 0)|\lambda ( d \gamma)
  \to 0, \qquad t\to 0^{+},
\end{equation}
holding for positive $f\in \mathcal{V}'$, see (\ref{2C}). By
(\ref{2c}) we then have
\begin{gather}
  \label{2y}
  I(t) \leq \int_{\Gamma} \left|\frac{1}{t} \int_0^t f(\gamma\cup x) dx - f(\gamma\cup 0) \right|\lambda ( d \gamma) + J(t), \\[.2cm] \nonumber J(t) := \int_{\Gamma}\left( \sum_{n=2}^{\infty} \frac{1}{t n!}\int_0^t \cdots \int_0^t f(\gamma_t \cup\{
  x_1, \dots , x_m\}) d x_1 \cdot d x_n \right) \lambda ( d \gamma).
\end{gather}
To estimate $J(t)$ we proceed as follows, cf. (\ref{3c}),
\begin{eqnarray*}
  J(t)& \leq & \int_{\Gamma}\left[ \sum_{n=0}^{\infty} \frac{1}{t n!}\int_0^t  \int_0^t \bigg{(} f(\gamma_t \cup\{x,y\} \cup
  x_1, \dots , x_m\}) d x_1 \cdot d x_n \bigg{)} d x d y \right] \lambda ( d \gamma) \\[.2cm] \nonumber& = & \frac{1}{t}\int_0^t  \int_0^t \bigg{(} \int_{\Gamma_t^c} \int_{\Gamma_t} f(\gamma \cup \xi \cup\{x,y\}) \lambda (d \gamma) \lambda ( d \xi) \bigg{)} d x dy
\\[.2cm] \nonumber& = & \frac{1}{t}\int_0^t  \int_0^t \bigg{(} \int_{\Gamma} \int_{\Gamma} f(\gamma \cup \xi \cup\{x,y\}) \mathds{1}_{\Gamma_t^c}(\gamma) \mathds{1}{\Gamma_c}(\xi) \lambda (d \gamma) \lambda ( d \xi) \bigg{)} d x dy \\[.2cm] \nonumber& = & \frac{1}{t} \int_0^t  \int_0^t \bigg{[} \int_{\Gamma} f(\gamma\cup\{x,y\}) \left( \sum_{\xi \subset \gamma} \mathds{1}_{\Gamma_t^c} (\gamma \setminus \xi) \mathds{1}_{\Gamma_t}(\xi) \right) \lambda ( d \gamma) \bigg{]} d x d y \\[.2cm] \nonumber& = & \frac{1}{t} \int_0^t  \int_0^t \bigg{(} \int_{\Gamma} f(\gamma\cup\{x,y\}) \lambda ( d \gamma) \bigg{)}  d x d y \leq t C'_f,
\end{eqnarray*}
where $C'_f$ is the constant in the estimate
\[
 \int_{\Gamma} f(\gamma\cup\{x,y\}) \lambda ( d \gamma) \leq C'_f,
\]
that can be obtained for a positive $f\in \mathcal{V}'$ similarly as
(\ref{2A}). Since $\gamma \mapsto f(\gamma\cup 0)$ is in
$\mathcal{W}$, the first term in the first line of (\ref{2y}) also
disappears in the limit $t\to 0^{+}$, which finally yields
(\ref{2z}).

Let us prove now that $\mathcal{V}=\mathcal{W}$. By Corollary
\ref{2pnco} (semigroup property and strong continuity) and by
(\ref{2D}) it follows that $f_t$ is differentiable in $t$ at all
$t\geq 0$ whenever $f\in \mathcal{V}$. Then $f_t \in \mathcal{W}$,
see (\ref{12a}). Let us prove that also $f_t\in \mathcal{V}$ in this
case. Indeed, by (\ref{2}) and (\ref{2A}) for $f\in
\mathcal{V}^{+}$, we have
\begin{eqnarray*}
 \int_{\Gamma} f_t (\gamma \cup 0 ) \lambda ( d \gamma) & = & \int_{\Gamma} \int_{\Gamma} f (\gamma \cup t \cup \xi) \mathds{1}_{\Gamma_t^c} (\gamma) \mathds{1}_{\Gamma_t}(\xi) \lambda ( d \gamma)\lambda ( d \xi) \\[.2cm] & = & \int_{\Gamma} f(\gamma \cup t) \lambda ( d\gamma) \leq C_f.
\end{eqnarray*}
Then for $f\in \mathcal{V}^{+}$, we can write
\begin{equation*}
 f_t = f + \int_0^t L^0 f_\tau d \tau_1,
\end{equation*}
from which we then obtain
\begin{equation}
 \label{2Y}
 \|f_t\| = \int_{\Gamma} f_t(\gamma) \lambda ( d \gamma) = \|f\| + \int_0^t \left(\int_{\Gamma} (L^0 f_\tau) (\gamma) \lambda ( d \gamma) \right)  d \tau.
\end{equation}
By Corollary \ref{2pnco} we know that $\|f_t\| = \|f\|$, which by
(\ref{2Y}) implies
\begin{equation}
 \label{2X}
 \forall f \in \mathcal{V}^{+} \qquad \varphi (L^0 f) := \int_{\Gamma} (L^0 f) (\gamma) \lambda ( d \gamma) = 0.
\end{equation}
Now we take $f\in \mathcal{W}^{+}$ and consider $\{f_m\}_{m\in
\mathds{N}} \subset \mathcal{V}$, where -- as above -- $f^{(n)}_m =
f^{(n)}$, $n\leq m$ and $f^{(n)}_m\equiv 0$ for $n>m$. Then
\[
\int_{\Gamma} f_m(\gamma\cup 0 ) \lambda ( d \gamma) = \int_{\Gamma}
( D f_m) (\gamma) \lambda ( d\gamma) \leq \|Df\|,
\]
by which and Lebesgue's dominated convergence theorem we conclude
that
\begin{equation}
  \label{0a}
  \int_{\Gamma} f(\gamma\cup 0 ) \lambda ( d \gamma)\leq \|D f\| < \infty,
\end{equation}
which by (\ref{2a}) yields $\mathcal{V}= \mathcal{W}$. By combining
(\ref{0a}) and (\ref{2X}) we obtain in turn
\begin{equation}
  \label{0c}
   \int_{\Gamma} f(\gamma\cup 0 ) \lambda ( d \gamma)= \|D f\|,
   \qquad f\in \mathcal{W}^{+} :=\mathcal{W}\cap \mathcal{X}^{+}.
\end{equation}
Thus, it remains to show that $(L^0,\mathcal{W})$ is closed. By
(\ref{0b}) and (\ref{0a}) it follows that the norms $\|\cdot
\|_{\mathcal{V}}$ and $\|\cdot \|_{\mathcal{W}}$ are equivalent, see
(\ref{12A}) and (\ref{2B}). Then the graph of $(L^0,\mathcal{W})$ is
closed in the graph norm, which yields the closedness and hence the
whole proof.
\end{proof}

\subsection{The model}

Our principal model is a modification of the soluble model just
described. Its main new aspect is that each particle by reaching the
edge divides at random into two progenies with randomly distributed
traits $x,y\in \bar{\mathds{R}}_{+}$. In addition, we assume here
that the particles can disappear (die) at random also outside of the
origin. The Fokker-Planck-Kolmogorov equation
\begin{equation}
  \label{Ma}
  \frac{d}{dt} f_t = L f_t, \qquad f_{t}|_{t=0} = f_0
\end{equation}
corresponding to this our model is defined by the Kolmogorov
operator $L$ that has the following form, cf. (\ref{2b}),
\begin{gather}
  \label{M2}
(Lf)(\gamma) = (Df)(\gamma) + m \int_{\mathds{R}_{+}} f
(\gamma\cup x) dx \\[.2cm] \nonumber - m|\gamma| f(\gamma) + \sum_{\{x,y\}\subset
\gamma} b(x,y) f (\gamma \setminus \{x,y\} \cup 0).
\end{gather}
Here $m\geq 0$ is the mortality rate and $b$ is a symmetric
probability density which hereby has the property
\begin{equation}
  \label{M3}
\frac{1}{2}  \int_{\bar{\mathds{R}}^2} b(x,y) d x d y = 1.
\end{equation}
For $\sigma>0$, set
\begin{equation}
 \label{sigma}
 \phi_\sigma (x) = (1+ x)^{-\sigma}, \qquad x\in \bar{\mathds{R}}_{+}.
\end{equation}
Our assumption concerning the cell cycle probability density is that
   \begin{equation}
     \label{M22}
   \forall x,y \qquad  b (x,y) \leq b^*\left[ \phi_{\sigma+1}(x)\phi_{\sigma}(y) + \phi_{\sigma}(x)\phi_{\sigma+1}(y) \right],
   \end{equation}
holding with some $\sigma \geq 3$ and $b^*>0$. Then (\ref{Ma}) with
$L$ given in (\ref{M2}) describes a drift of the particles towards
the origin (with unit speed) subject to a random death that occurs
at $x\in \bar{\mathds{R}}_{+}$ with constant rate $m$. At the
origin, the particle produces two progenies whose initial traits
(times to their division) are random. According to (\ref{M2}) the
dynamics of the considered model is characterized by the following
competing processes: (a) disappearance of the existing particles at
the edge $x=0$ and due to the mentioned random death; (b) appearance
of new particles in the course of division. It is quite clear that,
for $m=0$, the branching is supercritical and thus the population
will grow ad infinitum.  Among our aims in this work is to find a
trade-off condition for these two processes that secures the
boundedness in time of the population mean size.

As mentioned above, our model is intended to capture the basic
aspects of the dynamics of a population of tumor cells consisting in
the following: (a) malfunctioning of regulatory mechanisms and hence
uncontrolled proliferation with random cycle length; (b) increased
mortality caused by therapeutics; (c) death occurring at random with
no inter-cell dependence. In the model, aspect (a) corresponds to
the independent division with random cycle length, for a given
particle equal to its trait $x$ at the moment of its appearance.
Aspects (b) and (c) are taken into account in the second and third
terms of $L$, see (\ref{M2}). The choice of the model parameters is
based on the following reasons: (a) we believe that the therapeutic
effect on a cell is nearly independent of its age (phase of
mitosis); (b) $b(x,y)$ is often modeled as the product of two
$\Gamma$-densities $x^k e^{-\alpha x}$, cf. \cite{Gabriel,Ye}, which
clearly satisfies (\ref{M22}). See  also \cite{Marz} for more on
cell cycle modeling and Section 4 below for further comments.

Let us now define $L$ as an operator in $\mathcal{X}$. Set
\begin{gather}
  \label{M6}
L = A + B = A + B_1 + B_2, \\[.2cm] \nonumber
(Af)(\gamma) = ( Df)(\gamma) - m|\gamma| f(\gamma) \\[.2cm]
\nonumber (B_1f)(\gamma) = \sum_{\{x,y\}\subset \gamma} b(x,y)
f(\gamma
\setminus \{x,y\}\cup 0), \\[.2cm]
\nonumber (B_2f)(\gamma) = m\int_{\mathds{R}_{+}}  f (\gamma \cup x)
d x.
\end{gather}
Note that both $B_i$ are positive. Set
\begin{equation}
  \label{M9}
  h_m(\gamma) = 1 + m|\gamma|, \qquad m>0.
\end{equation}
By (\ref{La}) and (\ref{M6}), (\ref{M9}) we then have
\begin{equation}
  \label{M12}
  \|B_2 f\| \leq \|f\|_{h_m}.
\end{equation}
At the same time, for $f\in \mathcal{W}^{+}$, we have, cf.
(\ref{La}),
\begin{gather}
  \label{M11}
\|B_1 f \| = \int_{\Gamma}\left( \sum_{\{x,y\} \subset \gamma}
b(x,y)
f(\gamma\setminus \{x,y\}\cup 0)\right) \lambda (d \gamma) \\[.2cm] \nonumber = \frac{1}{2} \int_{\Gamma}\left( \sum_{x\in \gamma} \sum_{y\in \gamma\setminus x} b(x,y) f(\gamma\setminus \{x,y\}\cup 0)\right) \lambda (d \gamma) \\[.2cm] \nonumber = \frac{1}{2}\int_{\Gamma} \left(
\int_{\bar{\mathds{R}}_{+}} \sum_{y\in \gamma} b(x,y)
f(\gamma\setminus y\cup 0) d x \right) \lambda ( d
\gamma) \\[.2cm] \nonumber= \frac{1}{2}\int_{\Gamma}    \left(
\int_{\bar{\mathds{R}}_{+}^2} b(x,y) d x dy\right) f(\gamma\cup 0)
\lambda ( d \gamma) = \|D f\|,
\end{gather}
where we have taken into account (\ref{M3}) and (\ref{0c}). Keeping
this and (\ref{M12}) in mind we set
\begin{equation}
  \label{M13}
\mathcal{D}(A)= \mathcal{W}\cap   \mathcal{X}_{h_m}, \qquad
\mathcal{D}^{+}(A)= \mathcal{D}(A) \cap \mathcal{X}^{+}.
\end{equation}
Note that, for $f\in \mathcal{X}_{h_m}$, $k_f$ defined in (\ref{9A})
is integrable on $\bar{\mathds{R}}_{+}$.  Then by (\ref{M12}) and
(\ref{M11}) we conclude that
\begin{equation}
  \label{M13a}
  B: \mathcal{D}(A) \to \mathcal{X}.
\end{equation}

\subsection{The result}

For positive $\varsigma$ and $\alpha$, we set
\begin{gather}
  \label{M27}
\psi_\alpha (x) = e^{-\alpha x} , \qquad
x\in\bar{\mathds{R}}_{+},\\[.2cm] \nonumber
  h_{\varsigma,\alpha} (\gamma) = 1 + \varsigma |\gamma| +
\sum_{x\in \gamma} \psi_\alpha
  (x), \qquad \gamma \in \Gamma.
\end{gather}
Next, assuming (\ref{M22}) holding with $\sigma \geq 3$, we
introduce
\begin{equation}
  \label{M27z}
  m_1 =\max\left\{ 0;\ \frac{\sigma - 1}{2 \sigma -5} \left(
\frac{b^*}{2}-\sigma\right) \right\}.
\end{equation}
Our result is formulated in the next statement where by a classical
solution of the Cauchy problem in (\ref{Ma}) with $f_0 \in
\mathcal{D}(A)$ -- as is standard for such problems \cite[Chapter
4]{Pazy} -- we mean a function $t\mapsto f_t\in
\overline{\mathcal{D}(A)}\subset \mathcal{X}$ which is: (a)
continuously differentiable at all $t\geq 0$; (b) such that both
equalities in (\ref{Ma}) are satisfied. Here
$\overline{\mathcal{D}(A)}$ denotes the domain of the closure of
$L=A+B$, see Lemma \ref{A2lm} below.
\begin{theorem}
  \label{1tm}
Assume that (\ref{M22}) holds with some $\sigma\geq 3$ and $b^*>0$.
Then, for each $m> m_1$ and $f_0\in \mathcal{D}^{+}_1 (A) := \{ f\in
\mathcal{D}^{+}(A): \|f\|=1\}$, the Fokker-Planck-Kolmogorov
equation (\ref{Ma}) has a unique classical positive solution $f_t$
such that $\|f_t\|=1$. Furthermore, there exists $m_2\geq m_1$
(explicitly computable) such that, for $m\geq m_2$, there exists
$\varsigma >0$ for which $\|f_t\|_{h_{\varsigma,\alpha}} \leq
\|f_0\|_{h_{\varsigma,\alpha}}$ for all $t>0$.
\end{theorem}
The proof of this theorem  will be performed in Section \ref{3S}
below. Here we make some comments to its results. The last part of
Theorem \ref{1tm} yields a balance condition between the
disappearance of the particles and the appearance of their
progenies. Indeed, the expected number of particles at time $t$ is
$N_1(t)$, see (\ref{3b}). By (\ref{M27}) and Theorem \ref{1tm} we
then have
\begin{equation}
  \label{M13c}
  N_1 (t) \leq \varsigma^{-1} \|f_0 \|_{h_{\varsigma,\alpha}},
\end{equation}
and thus $N_1(t)$ remains bounded if the mortality rate $m$ is
bigger than a certain quantity, explicitly computable in terms of
the cell cycle parameters, see (\ref{P}) below. Another conclusion
of this sort is that the evolution described by Theorem \ref{1tm} is
honest, cf. \cite{hon}, since the norm of $f_t$ is preserved, i.e.,
$\|f_t\|=1$. This, in particular, means that the system of particles
remains almost surely finite for all $t>0$. Indeed, since $f_t$ is
the Radon-Nikodym derivative of the state at time $t$, the fact that
$\|f_t\|<1$ would mean that the population is finite with
probability strictly less than one, and hence the estimate in
(\ref{M13c}) holds provided the system is finite. Since (perhaps)
Reuter's seminal paper \cite{Reuter}, the evolution of this kind is
called dishonest. More on the \emph{honesty theory} can be found in
\cite{hon}.

The backward Kolmogorov equation
\begin{equation}
  \label{M4}
  \frac{d}{dt}F_t = L^* F_t, \qquad F_t|_{t=0} = F_0,
\end{equation}
is dual to (\ref{Ma}) in the sense that
\[
\int_{\Gamma} F(\gamma) (Lf)(\gamma) \lambda ( d \gamma) =
\int_{\Gamma} (L^*F)(\gamma) Lf(\gamma) \lambda ( d \gamma).
\]
Here $F_t:\Gamma \to \mathds{R}$ is an \emph{observable} and
\begin{eqnarray}
  \label{M5}
  ( L^* F)(\gamma)& = & - ( D F)(\gamma) + \sum_{x\in \gamma} m(x)
  \left[ F(\gamma\setminus x) - F(\gamma)\right]\\[.2cm] \nonumber &
  + & \frac{1}{2}\sum_{x\in \gamma} \delta (x) \int_{\bar{\mathds{R}}_{+}^2} b(y,z) \left[
  F(\gamma \setminus x\cup \{y,z\}) - F(\gamma)
  \right] dy d z,
\end{eqnarray}
that additionally clarifies the nature of the dynamics described by
$L$ and its dual $L^*$.

\section{The Proof}

\label{3S}

The proof will be divided into two parts. First we construct a $C_0$
semigroup $S=\{S(t)\}_{t\geq 0}$ such that the solution in question
is obtained in the form $f_t = S(t)f_0$, for all $t\geq 0$ and
initial $f_0$ belonging to the domain of the generator of $S$. A
special attention here will be paid to proving that $S$ is
stochastic. In the second part, we prove the stated boundedness that
implies (\ref{M13c}).

\subsection{The stochastic semigroup}

The construction of the mentioned semigroup $S$ is based on a
perturbation technique, developed in \cite{TV}, and some aspects of
the honesty theory \cite{hon,Mustapha}. Its adaptation to the
present context is given in the following three statements. Therein,
we deal with a Banach space $\mathcal{E}$ equipped with a cone of
positive elements, $\mathcal{E}^{+}$, that have the following
property. There exists a positive linear functional,
$\varphi_{\mathcal{E}}$, such that $\|u\|_{\mathcal{E}} =
\varphi_{\mathcal{E}}(u)$ whenever $u\in \mathcal{E}^{+}$. Thereby,
the norm $\| \cdot \|_{\mathcal{E}}$ is additive on
$\mathcal{E}^{+}$.
\begin{proposition}\cite[Theorem 2.2]{TV}
  \label{A1ln}
Let  $(A,\mathcal{D}_A)$ be the generator of a substochastic
semigroup, $T_0=\{T_0(t)\}_{t\geq 0}$. Let also $B: \mathcal{D}_A\to
\mathcal{E}$ be positive and such that $\varphi_{\mathcal{E}} (
(A+B)u) \leq 0$ for all $u\in \mathcal{D}_A^{+}:=\mathcal{D}_A \cap
\mathcal{E}^{+}$. Then, for each $r\in (0,1)$, the operator $A+rB$
generates a substochastic semigroup, $T_r=\{T_r(t)\}_{t\geq 0}$.
Furthermore, there exists a substochastic semigroup,
$T_1=\{T_1(t)\}_{t\geq 0}$, on $\mathcal{E}$ such that $\|T_1(t) u -
T_r (t) u\|_{\mathcal{E}} \to 0$ as $r\to 1^{-}$, for all $u\in
\mathcal{E}$ and uniformly in $t$ on each $[0,T]$, $T>0$. The
semigroup $T_1$ is generated by an extension of
$(A+B,\mathcal{D}_A)$.
\end{proposition}
This statement is just an extended version of the celebrated Kato
perturbation theorem, cf. \cite[Sect. 2]{hon}. The semigroup $T_1$
may not be stochastic even if $\varphi_{\mathcal{E}} ( (A+B)u) = 0$.
In this case, $\|T_1(t) u\|_{\mathcal{E}}< \|u\|_{\mathcal{E}}$,
that  is, the evolution is dishonest. In order to establish the
honesty of $T_1$, one has to get additional information on its
properties. The first statement in this direction is a simple
consequence of Theorem 3.5 and Corollary 3.6 of \cite{hon}.
\begin{proposition}
  \label{A2pn}
The semigroup $T_1$ mentioned in Proposition \ref{A1ln} is honest if
and only if its generator is the closure of $(A+B,\mathcal{D}_A)$.
\end{proposition}
A more specific fact - applicable in $L^1$ spaces -- is provided by
the following statement.
\begin{proposition}
 \label{Mustaphapn} \cite[Theorem 2, page 156]{Mustapha}
In the setting of Proposition \ref{A1ln}, assume that $\mathcal{E}=
L^1 (\Omega, \nu)$ for appropriate $\Omega$ and $\nu$. Let there
exist $v\in \mathcal{D}_A$ such that: (a) $v$ is strictly positive;
(b) $(A+B)v \leq 0$. Both (a) and (b) hold $\nu$-almost everywhere
on $\Omega$. Then the generator of $T_1$ is the closure of
$(A+B,\mathcal{D}_A)$, and hence $T_1$ is honest -- by Proposition
\ref{A2pn}.
 \end{proposition}
Now we can turn to our models. For $\varepsilon \in (0,1)$ and $A$
and $B$ as in (\ref{M6}), we set
\begin{equation}
  \label{M33}
L^\varepsilon = A + (1-\varepsilon) B.
\end{equation}
Recall that the domain of both $A$ and $B$ is $\mathcal{D}(A)$
defined in (\ref{M13}).
\begin{lemma}
  \label{A1lm}
For each $\varepsilon \in (0,1)$, the operator $(L^\varepsilon,
\mathcal{D}(A))$ generates a substochastic semigroup, $S^\varepsilon
= \{S^\varepsilon(t)\}_{t\geq 0}$. Furthermore, there exists a
substochastic semigroup, $S = \{S(t)\}_{t\geq 0}$, on $\mathcal{X}$
such that $S^\varepsilon (t) \to S(t)$ as $\varepsilon \to 0$,
strongly and uniformly on $[0,T]$, $T>0$. The semigroup $S$ is
generated by an extension of the operator $(L,\mathcal{D}(A))$.
\end{lemma}
\begin{proof}
The operator $(A,\mathcal{D}(A))$ generates a substochastic
semigroup, $S^0$, with
\begin{equation*}
( S^0 (t) f)(\gamma) = \exp\left( - t m |\gamma|\right)f(\gamma_t),
\end{equation*}
see (\ref{12a}). Obviously, $B$ is a positive operator; hence,
$B:\mathcal{D}^{+}(A)\to \mathcal{X}^{+}$, see (\ref{M13a}).  By
(\ref{2b}), (\ref{2X}), and then by (\ref{M6}), for $f\in
\mathcal{D}^{+}(A)$, we obtain
\begin{eqnarray*}
\varphi((A+B)f)& = & \int_{\Gamma} (D f)(\gamma) \lambda ( d \gamma)
+ \int_{\Gamma} \left(\sum_{\{x,y\} \subset \gamma} b(x,y)
f(\gamma\setminus \xi \cup 0) \right) \lambda ( d \gamma) \qquad \\[.2cm] \nonumber &
- & m\int_{\Gamma} |\gamma| f(\gamma) \lambda ( d \gamma) +
 m \int_{\Gamma} \int_{\bar{\mathds{R}}_{+}} f(\gamma\cup x) d x
\lambda (d \gamma) \\[.2cm] \nonumber & = &
 \int_{\Gamma} (D f)(\gamma) \lambda ( d \gamma) +
 \frac{1}{2}\int_{\Gamma}\left(\int_{\bar{\mathds{R}}_{+}^2}  b(x,y)
 d x d y \right)
f(\gamma \cup 0) \lambda ( d \gamma)  \\[.2cm] \nonumber
& - & m\int_{\Gamma} |\gamma| f(\gamma) \lambda ( d \gamma) +
\int_{\Gamma} \left(\sum_{x\in \gamma} m \right) f(\gamma)
\lambda ( d \gamma)\\[.2cm] \nonumber & = & \int_{\Gamma} (L^0
f)(\gamma) \lambda ( d \gamma) = 0,
\end{eqnarray*}
see (\ref{2b}) and (\ref{M3}). Since $B$ is positive, this yields
that, for each $\varepsilon \in (0,1)$ and $f\in
\mathcal{D}^{+}(A)$, the following holds $\varphi (L^\varepsilon f)
\leq 0$. Then $(L^\varepsilon, \mathcal{D}(A))$, see (\ref{M33}),
generates $S^\varepsilon$ as stated, and the semigroup $S$ is
obtained in accordance with Proposition \ref{A1ln}.
\end{proof}
\begin{lemma}
 \label{A2lm}
Let $m_1$ be as in (\ref{M27z}). Then, for $m> m_1$, the semigroup
$S$ constructed in Lemma \ref{A1lm} is generated by the closure of
$(A+B,\mathcal{D}(A))$ and hence is honest therefore.
\end{lemma}
\begin{proof}
Here we employ Proposition \ref{Mustaphapn}. To this end we
introduce $v\in \mathcal{D}(A)$ by the following expression
 \begin{equation*}
  v(\gamma) = |\gamma|! \prod_{x\in \gamma} \phi_\sigma (x), \qquad \sigma \geq 3,
 \end{equation*}
where $\phi_\sigma$ is as in (\ref{sigma}). It is clearly strictly
positive everywhere on $\Gamma$. Let us show that $v\in
\mathcal{D}(A)$, see (\ref{M13}). By (\ref{11}) and then by
(\ref{L1}) and (\ref{La}) we obtain
\begin{gather*}
  \|D v\| = \sigma \int_{\Gamma} |\gamma|! \sum_{x\in \gamma}
  \phi_{\sigma + 1} (x) \prod_{y\in \gamma\setminus x} \phi_\sigma
  (y) \lambda ( d \gamma)\\[.2cm]  = \sigma \int_{\Gamma}
  (|\gamma|+1)! \left(\int_{0}^{+\infty} \phi_{\sigma +1}(x) d x
  \right) \prod_{y\in \gamma} \phi_\sigma
  (y) \lambda ( d \gamma)\\[.2cm]  = \sum_{n=0}^{+\infty} (n+1)
  (\sigma - 1)^{-n} =
  \Big( \frac{\sigma-1}{\sigma -2} \Big)^2< \infty.
\end{gather*}
Hence, $v\in \mathcal{W}$. Likewise,
\begin{gather*}
  \|v \|_{h_m} = \sum_{n=0}^{+\infty} (\sigma-1)^{-n} + m \sum_{n=1}^{+\infty} n (\sigma - 1)^{-n} =
\frac{\sigma-1}{\sigma-2} + m \frac{\sigma -1}{(\sigma-2)^2} <
\infty,
\end{gather*}
that yields $v\in \mathcal{D}(A)$. Thus, to apply Proposition
\ref{Mustaphapn} we have to show that
\begin{equation}
  \label{P8}
 \forall \gamma \in \Gamma \qquad   (A v) (\gamma) + (Bv)(\gamma) \leq 0.
\end{equation}
For $\gamma = \varnothing$, both terms on the left-hand side of
(\ref{P8}) vanish. For $\gamma=\{x\}$,
\[
{\rm LHS(\ref{P8})} = - \sigma \phi_{\sigma + 1} (x) -
m\frac{\sigma-3}{\sigma -1} < 0,
\]
whenever $\sigma \geq 3$. For $|\gamma|\geq 2$, we have
\[
(A v)(\gamma) = - \sigma |\gamma|! \sum_{x\in \gamma} \phi_{\sigma +
1} (x) \prod_{y\in \gamma\setminus x} \phi_\sigma (y) - m |\gamma|
|\gamma|!  \prod_{y\in \gamma} \phi_\sigma (y).
\]
Now by (\ref{M22}), we obtain
\begin{gather*}
(B v)(\gamma) = (|\gamma|-1)! \sum_{\{x,y\}\subset \gamma} b(x,y)
\prod_{z\in \gamma\setminus \{x,y\}} \phi_\sigma (z) + (|\gamma|+1)! \frac{m}{\sigma - 1} \prod_{x\in \gamma} \phi_\sigma (x) \\[.2cm]
\nonumber \leq (|\gamma|- 1)! (|\gamma|-1) b^* \sum_{x\in \gamma}
\phi_{\sigma + 1} (x) \prod_{y\in \gamma\setminus x} \phi_\sigma (y)
+( |\gamma|+1)! \frac{m}{\sigma - 1} \prod_{x\in \gamma} \phi_\sigma
(x).
\end{gather*}
Then
\begin{gather}
  \label{P10}
{\rm LHS(\ref{P8})} \leq - \left(\sigma - \frac{b^*}{2}
\right)|\gamma|! \sum_{x\in \gamma} \phi_{\sigma+1}(x) \prod_{y\in
\gamma\setminus x} \phi_\sigma (y)  - m \frac{2\sigma - 5}{\sigma-1}
|\gamma||\gamma|! \prod_{x\in \gamma} \phi_\sigma (x).
\end{gather}
If $b(x,y)$ is such that $b^*\leq 2\sigma$, we take $m_1=0$ and
obtain (\ref{P8}). For $b^*\leq 2\sigma$, we use the fact that
$\phi_{\sigma+1} (x) \leq \phi_\sigma (x)$, $x\geq 0$, and then get
from (\ref{P10}) the following
\[
{\rm LHS(\ref{P8})} \leq \left( \frac{b^*}{2} - \sigma -
\frac{2\sigma - 5}{\sigma -1}m\right) |\gamma||\gamma|!  \prod_{x\in
\gamma} \phi_\sigma (x) \leq 0,
\]
where the latter inequality holds in view of the assumed $m\geq
m_1$, see (\ref{M27z}).
\end{proof}

\subsection{The boundedness}

To prove the boundedness which yields (\ref{M13c})  we are going to
employ another statement of \cite{TV}. Thus, in the context of
Proposition \ref{A1ln} we further impose the following.
\begin{assumption}
  \label{Aass}
There exists a linear subspace, $\widetilde{\mathcal{ E}} \subset
\mathcal{ E}$, which has the following properties:
\begin{itemize}
  \item[(i)] $\widetilde{\mathcal{ E}}$ is dense in $\mathcal{ E}$ in the norm
  $\|\cdot\|_{\mathcal{ E}}$.
  \item[(ii)] There exists a norm, $\|\cdot\|_{\widetilde{\mathcal{ E}}}$, on $\widetilde{\mathcal{ E}}$ that makes
  it a Banach space and the  embedding $\widetilde{\mathcal{E}}$ into $\mathcal{E}$ is continuous.
  \item[(iii)] $\widetilde{\mathcal{ E}}^{+}:= \widetilde{\mathcal{ E}}\cap \mathcal{ E}^{+}$ is a generating cone
  in $\widetilde{\mathcal{ E}}$. The norm $\|\cdot\|_{\widetilde{\mathcal{ E}}}$ is additive on $\widetilde{\mathcal{ E}}^{+}$
  and hence there exists a linear functional, $\varphi_{\widetilde{\mathcal{
  E}}}$, such that $\|u\|_{\widetilde{\mathcal{ E}}}= \varphi_{\widetilde{\mathcal{
  E}}}(u)$ whenever $u\in\widetilde{\mathcal{
  E}}$.
\item[(iv)] The cone $\widetilde{\mathcal{ E}}^{+}$ is dense in
$\mathcal{ E}^{+}$.
\end{itemize}
\end{assumption}
For $(A,\mathcal{D}_A)$ as in Proposition \ref{A1ln}, set
$\widetilde{\mathcal{D}}_A= \{ u \in \mathcal{D}: A u \in
\widetilde{\mathcal{E}}\}$.
\begin{proposition}\cite[Theorem 2.6]{TV}
  \label{A3pn}
Let the assumption of Proposition \ref{A1ln} be satisfied. Assume
also that $\widetilde{\mathcal{ E}}$ is a subspace of $\mathcal{ E}$
which satisfies Assumption \ref{Aass}. Additionally, assume that
\begin{itemize}
  \item[(a)] The restrictions $T_0(t)|_{\widetilde{\mathcal{ E}}}$
  constitute a $C_0$-semigroup in the norm of $\widetilde{\mathcal{
  E}}$, generated by
  $(A,\widetilde{\mathcal{D}}_A)$.
\item[(b)] $B: \widetilde{\mathcal{D}}_A \to
\widetilde{\mathcal{E}}$.
\item[(d)] The following
holds: $\varphi_{\widetilde{\mathcal{E}}} ((A+B)u) \leq 0$.
\end{itemize}
Then the the semigroup $T_1 = \{T_1(t)\}_{t\geq 0}$ from Proposition
\ref{A1ln} leaves $\widetilde{\mathcal{E}}$ invariant.  The
restrictions $T_1(t)|_{\widetilde{\mathcal{E}}}$ constitute a
substochastic semigroup on $\widetilde{\mathcal{E}}$.
\end{proposition}
\vskip.1cm \noindent {\it Proof of Theorem \ref{1tm}.} In view of
\cite[Theorem 1.3, page 102]{Pazy}, for $m> m_1$ the existence and
uniqueness of the solution in question follows by the existence and
the properties of the semigroup $S$ obtained in Lemma \ref{A2lm}.
That is, it has the form $f_t = S(t) f_0$. Let us prove the second
part of the theorem. To this end, we employ Proposition \ref{A3pn},
where as $\widetilde{\mathcal{E}}$ we take
$\mathcal{X}_{h_{\varsigma,\alpha}} =\mathcal{X}_{h_m}$. Note that
the latter means that they are equal as sets, and that $m>m_1$ is
positive even if $m_1=0$, see (\ref{M27z}). Clearly,
$\mathcal{X}_{h_{\varsigma,\alpha}}$ has all the properties as in
Assumption \ref{Aass}, cf. (\ref{La1}). Moreover, in this case
$\widetilde{\mathcal{D}}_A = \{ f \in \mathcal{D}(A) \cap
\mathcal{X}_{h_m} : Lf \in \mathcal{X}_{h_m}\}$. By direct
inspection one checks that both (a) and (b) assumed in Proposition
\ref{A3pn} are satisfied for this choice of
$\widetilde{\mathcal{E}}$. Now we prove that (c) also holds for
$m\geq m_2$ where the latter has to be found. For $f\in
\widetilde{\mathcal{D}}_A \cap
\mathcal{X}^{+}_{h_{\varsigma,\alpha}}$, we write, cf. (\ref{M4}),
(\ref{M5}),
\begin{eqnarray}
  \label{P11}
  \varphi_{h_{\varsigma,\alpha}} ((A+B) f) & = & \int_{\Gamma} h_{\varsigma,
  \alpha}(\gamma) ((A+B) f)(\gamma) \lambda ( d \gamma) \\[.2cm]
  \nonumber & = & \int_{\Gamma} (L^* h_{\varsigma,
  \alpha})(\gamma)  f(\gamma) \lambda ( d \gamma) \\[.2cm]
  \nonumber & = & - \int_{\Gamma}\left(\sum_{x\in \gamma}\left[m \varsigma + (m-\alpha) \psi_\alpha (x)  \right]
  \right)f(\gamma) \lambda ( d \gamma) \\[.2cm]
  \nonumber & + & \Upsilon (\varsigma, \alpha) \int_{\Gamma}
  f(\gamma\cup 0) \lambda ( d \gamma),
\end{eqnarray}
where
\begin{gather}
  \label{P12}
\Upsilon (\varsigma, \alpha)  = \varsigma - 1 + \hat{\beta}(\alpha),
\\[.2cm] \nonumber \hat{\beta} (\alpha) = \int_{\bar{\mathds{R}}_{+}}
\beta (x) e^{-\alpha x} d x, \qquad \beta (x) =
\int_{\bar{\mathds{R}}_{+}} b(x,y) dy.
\end{gather}
Since $\beta$ is integrable, by the Riemann-Lebesgue lemma it
follows that $\hat{\beta}(\alpha) \to 0$ as $\alpha \to +\infty$.
Then one finds $\alpha>0$ such that $\hat{\beta}(\alpha) < 1$. For
this $\alpha$, $\varsigma = 1 - \hat{\beta} (\alpha)$ is positive
and thus can be used in (\ref{P11}), where it yields $\Upsilon
(\varsigma,\alpha) \leq 0$. Now, for this $\alpha$ and $m_1$ as in
(\ref{M27z}), we set
\begin{equation}
 \label{P}
 m_2 = \min\{m_1; \alpha\},
\end{equation}
that, for $m\geq m_2$, yields
\[
{\rm LHS(\ref{P11})} \leq 0,
\]
which then by Proposition \ref{A3pn} yields in turn
\[
N_1(t) \leq \frac{\|f_t\|_{h_{\sigma,\alpha}}}{1-\hat{\beta}(m_2)}
\leq \frac{\|f_0\|_{h_{\sigma,\alpha}}}{1-\hat{\beta}(m_2)} . \qquad
\varsigma = 1 - \hat{\beta}(\alpha)\leq 1 - \hat{\beta}(m_2).
\]
This completes the proof of Theorem \ref{1tm} with $m_2$ defined in
(\ref{P}). \hfill $\square$

\section{Summary and Concluding Remarks}

We begin by making a brief summary of the aspects of the theory
presented here, understandable also for non-mathematicians. Then we
discuss some aspects of this work, as well as outline its possible
continuation.

\subsection{Summary}

In cancer biology, it is well established that cancer cells
proliferate wildly by repeated, uncontrolled mitosis. "Unlike normal
cells, cancer cells ignore the usual density-dependent inhibition of
growth ... piling up until all nutrients are
exhausted\footnote{https://www.biology.iupui.edu/biocourses/N100H/ch8mitosis.html}"
Therefore, to model populations of cancer cells one can use
`particles' that undergo independent branching into two new
`particles' after some random time. Being unharmed populations of
such `particles' grow ad infinitum since the branching number is
two. Therapeutic pressure causes disappearance of some of them from
the population before branching, the result of which may be
restricting the population growth. The effect of the treatment is
proportional to its intensity and to the mean length of the
inter-mitosis period, during which it acts. Then the paramount
problem of modeling of such populations is to find qualitative
relations between the treatment intensity and probabilistic
parameters of the cell cycle processes in a given population. In
this work, we find this relation in the form $m\geq m_2$ with
$m_2>0$ defined in (\ref{P}) and (\ref{M27z}).

\subsection{The model and its study}
The proposed model seems to be the simplest individual-based model
that takes into account the basic aspects of the phenomenon which we
intended to describe: (a) essential mortality caused by external
factors and independent of the interactions inside the population;
(b) randomly distributed lifetimes of the population members, at the
end of which each of them branches into two progenies; (c) branching
independent of the interactions inside the population. The main
difficulty of its mathematical study stemmed from the presence of
the gradient in the Kolmogorov operator $L$ in (\ref{M2}), which is
typical for transport problems \cite{Mustapha}. A more general
version of the proposed model instead of the  last summand in
(\ref{M2}) could contain
\[
\sum_{\xi \subset \gamma} b(\xi) f(\gamma\setminus \xi \cup 0),
\qquad \xi \in \Gamma,
\]
that corresponds to branching into a `cloud' $\xi$ with possibly
random number of progenies. In fact, this might be done in the
present context at the expense of a modification of the bound in
(\ref{M22}). Our choice was motivated by the reasons of simplicity
and practical applications -- mitosis with two progenies.
Noteworthy, in our model the lifetimes of siblings are in general
dependent as random variables. The independent case would correspond
to the choice $b(x,y) = \beta (x) \beta (y)$ with $\beta$ as in
(\ref{P12}). Note, however, that the definition of $m_2$ in
(\ref{P}) remains the same in this case. In order to take into
account also dependence like `parent-progeny', cf \cite{Marz}, one
would make the trait more complex by including the corresponding
parameter. For example, instead of $\bar{\mathds{R}}_{+}$ one may
take $\bar{\mathds{R}}_{+}^2$ consisting of pairs $\hat{x}=(x,y)$ in
which $x$ is still time to fission whereas $y$ is responsible for
the mentioned dependence. This additional trait can be used to
model, e.g., further mutations of the tumor cells.

\subsection{The practical meaning}
As mentioned above, we believe that the proposed theory can have
direct practical applications for the following reasons. There
exists a rich literature on modeling -- parameter fitting including
-- of various types of cancer, see e.g.,
\cite{Marz,Gabriel,Tyson,Ye} and the sources quoted in these
publications. This means that, in a concrete situation, one can
calculate $m_2$ by means of (\ref{P}) and (\ref{M27z}), which then
can be used to estimate the corresponding therapeutic dose.

\subsection{Further development}
Along with the modifications of the model already mentioned above in
this section, we plan to consider also its version describing
infinite populations. Here we plan to employ methods developed in
\cite{Koz}, of which studying finite populations is a part. We also
plan to develop a mesoscopic theory of this model by means of
scaling techniques and Poisson approximations, see also \cite{Koz}.
This would allow for connecting the microscopic theory developed in
this way to a description based on aggregate parameters, similar to
that is \cite{Mee,Rot}.

\section*{acknowledgements} The author thanks Michael R\"ockner for
valuable discussions. He also thanks Marek Kimmel for pointing to
the results published in \cite{Marz} and providing the text of this
publication.

%
%



\end{document}